\documentclass{article}
\usepackage{amsmath,amssymb,amsthm,url}
\newtheorem{conjecture}{Conjecture}[section]
\newtheorem{proposition}[conjecture]{Proposition}
\newtheorem{theorem}[conjecture]{Theorem}
\newtheorem{corollary}[conjecture]{Corollary}
\theoremstyle{definition}
\newtheorem{example}[conjecture]{Example}
\newtheorem{definition}[conjecture]{Definition}

\title{Power residue symbols and the exponential local-global principle}
\author{Henry (Maya) Robert Thackeray (ORCID: 0000-0003-4467-9070)\\Department of Mathematics and Applied Mathematics,\\University of Pretoria, Pretoria, 0002 South Africa\\maya.thackeray@up.ac.za}
\date{}
\begin{document}
\maketitle
Preprint of an article submitted for consideration in International Journal of Number Theory \textcopyright{} 2025 [copyright World Scientific Publishing Company] \url{https://worldscientific.com/worldscinet/ijnt} (journal allows posting of preprints at any time in not-for-profit subject-based preprint servers; formatting of the current document is different from formatting of preprint submitted to the journal; copyright for preprint submitted to journal is transferred to World Scientific on acceptance for publication)

\mbox{}

\begin{abstract}
The exponential local-global principle, or Skolem conjecture, says: Suppose that \(b\) is a positive integer, and that the sequence \((u_{n})_{n = -\infty}^{\infty}\) is such that every term is in \(\mathbb{Z}[1/b]\), the linear recurrence \(u_{n + d} = a_{1}u_{n + d - 1} + \cdots + a_{d}u_{n}\) holds for all integers \(n\), and every root of \(x^{d} - a_{1}x^{d - 1} - a_{2}x^{d - 2} - \cdots - a_{d}\) is nonzero and simple; then there is no zero term \(u_{n}\) if and only if, for some integer \(m\) that is larger than \(1\) and relatively prime to \(b\), every term \(u_{n}\) is not in \(m\mathbb{Z}[1/b]\).

Particular cases of the conjecture are known, but the general conjecture is open. This paper proves some apparently new quadratic and degenerate cubic cases of the exponential local-global principle via power residue symbols.

This work was presented at the Stellenbosch Number Theory Conference 2025 in January 2025 at Stellenbosch University; much of the work was also presented at the 67th Annual Congress of the South African Mathematical Society in December 2024 at the University of Pretoria.
\end{abstract}

Keywords: Recurrence; local-global principle; power residue symbol; number theory.

Mathematics Subject Classification 2020: 11B37

\section{Introduction}

The exponential local-global principle, also called the Skolem conjecture, says (see \cite{Letal22}):

\begin{conjecture}[Exponential local-global principle]
Let \(a_{1}, \ldots, a_{d} \in \mathbb{Q}\) be such that \(a_{d} \neq 0\) and all roots of the polynomial \(f(x) = x^{d} - a_{1}x^{d - 1} - a_{2}x^{d - 2} - \cdots - a_{d}\) are simple. Suppose that \(b\) is a positive integer, and that the sequence \(u = (u_{n})_{n = -\infty}^{\infty}\) is such that every term is in \(\mathbb{Z}[1/b]\) and, for all \(n \in \mathbb{Z}\),
\[u_{n + d} = a_{1}u_{n + d - 1} + \cdots + a_{d}u_{n}.\]
There is no zero term \(u_{n}\) if and only if, for some integer \(m \geqslant 2\), we have \(\gcd(b, m) = 1\) and every term \(u_{n}\) is not in \(m\mathbb{Z}[1/b]\).\hfill\(\square\)
\end{conjecture}

If there is a zero term in \(u\), then the condition about \(m\) must clearly fail. The interesting part of the conjecture is the inverse -- the statement that if there is no zero term in \(u\), then the condition about \(m\) holds. The conjecture is open in general.

The assumption that the roots of \(f(x)\) are simple is needed, as the following example illustrates.

\begin{example}
(See \cite{Letal22}.) Consider the sequence \(u\) in \(\mathbb{Z}[1/2]\) given by
\[u_{n} = (2n + 1)2^{n} \textrm{ for } n \in \mathbb{Z}.\]
We have \(u_{n + 2} = 4u_{n + 1} - 4u_{n}\) for \(n \in \mathbb{Z}\), and
\[f(x) = x^{2} - 4x + 4 = (x - 2)^{2}.\]
No term of \(u\) is zero. However, for each odd positive integer \(m\), we have \(m \mid u_{n}\) for every integer \(n\) of the form \(n = km + (m - 1)/2\) with \(k \in \mathbb{Z}\).\hfill\(\square\)
\end{example}

For each prime \(p\), write \(v_{p}(x)\) for the \(p\)-adic valuation of a nonzero rational number \(x\) (for example, \(v_{p}(p^{n}) = n\) for \(n \in \mathbb{Z}\)). We say that a prime \(p\) is \emph{coprime} to a nonzero rational number \(x\) if and only if \(v_{p}(x) = 0\).

It is straightforward to prove the exponential local-global principle in the linear case, that is, in the case where \(d = 1\).

\begin{proposition}
The exponential local-global principle holds in the case where \(d = 1\).
\end{proposition}
\begin{proof}
Consider a sequence \(u\) of numbers in \(\mathbb{Z}[1/b]\) where, for some \(a_{1} \in \mathbb{Q} - \{0\}\), we have \(u_{n + 1} = a_{1}u_{n}\) for \(n \in \mathbb{Z}\). Clearly, \(u_{n} = u_{0}a_{1}^{n}\) for all \(n \in \mathbb{Z}\). If \(u_{0} = 0\), then all terms of \(u\) are zero and the conjecture holds. If \(u_{0} \neq 0\), then no term of \(u\) is zero and each prime \(p\) that is coprime to both \(a_{1}\) and \(u_{0}\) satisfies \(u_{n} \not\equiv 0 \pmod{p}\) for all \(n \in \mathbb{Z}\), so the conjecture holds (and all but finitely many primes work as choices for \(m\)).
\end{proof}

The major goal of this paper is to prove a strengthening of the exponential local-global principle in many cases where \(d = 2\) and \(f(x)\) has two distinct rational roots. From this, some degenerate cases where \(d = 3\) will quickly follow.

This work was presented at the Stellenbosch Number Theory Conference 2025 in January 2025 at Stellenbosch University; much of the work was also presented at the 67th Annual Congress of the South African Mathematical Society in December 2024 at the University of Pretoria. This paper covers the first year of a planned three-year research project associated with the University of Pretoria's Research Development Programme; for the second and third years of the project, the aim is to prove the exponential local-global principle in other cases (including families of cases where \(f(x)\) is cubic or quartic), and to examine whether and when the principle holds if algebraic number fields other than \(\mathbb{Q}\) are considered.

In this article, from now on until the section on degenerate cubic cases, we assume that \(d = 2\), we assume that the numbers \(a_{1}, a_{2}\) in \(\mathbb{Q}\) are such that \(a_{2} \neq 0\) and \(a_{1}^{2} + 4a_{2} \neq 0\), we let \(b\) be a positive integer, we consider a sequence \(u = (u_{n})_{n = -\infty}^{\infty}\) of numbers in \(\mathbb{Z}[1/b]\) such that
\[u_{n + 2} = a_{1}u_{n + 1} + a_{2}u_{n}\]
for all \(n \in \mathbb{Z}\), and we assume that the roots \(c_{1}\) and \(c_{2}\) of \(f(x) = x^{2} - a_{1}x - a_{2}\) are rational. By Viete's formulas, \(c_{1}c_{2} = -a_{2} \neq 0\). The condition \(a_{1}^{2} + 4a_{2} \neq 0\) yields \(c_{1} \neq c_{2}\), so by a well known result, there are \(b_{1}, b_{2} \in \mathbb{Q}\) such that, for \(n \in \mathbb{Z}\), we have \(u_{n} = b_{1}c_{1}^{n} + b_{2}c_{2}^{n}\). If \(0 \in \{b_{1}, b_{2}\}\) then \(u \equiv 0\) or we are in a \(d = 1\) case; from now on, let us assume \(b_{1} \neq 0\) and \(b_{2} \neq 0\). We write \(B = -b_{2}/b_{1}\) and \(C = c_{1}/c_{2}\).

Under those assumptions, we shall prove the following main theorem.

\begin{theorem}\label{thmain}
Suppose that
\begin{itemize}
\item[(1)] We have \(C^{m} = B\) for some \(m \in \mathbb{Z}\), or
\item[(2)] We have \((C, B) = (1, -1)\), or
\item[(3)] We have \(B^{k}C^{\ell} \notin \{\pm 1\}\) for all \(k \in \mathbb{Z} - \{0\}\) and all \(\ell \in \mathbb{Z}\).
\end{itemize}
There is no zero term in \(u\) if and only if, for infinitely many primes \(p\), we have \(\gcd(b, p) = 1\) and every term \(u_{n}\) is not divisible by \(p\).\hfill\(\square\)
\end{theorem}

To prove Theorem \ref{thmain}, we shall use a Chebotarev-density-like result of Schinzel \cite{S77} on power residue symbols. Unlike some articles of Schinzel \cite{S77, S03} that proved some quadratic cases of the exponential local-global principle, the number \(a_{2}\) is not necessarily \(\pm 1\) in our work; Theorem \ref{thmain} is apparently a new result.

\section{Initial Results}

We have \(B = -b_{2}/b_{1} \in \mathbb{Q} - \{0\}\) and \(C = c_{1}/c_{2} \in \mathbb{Q} - \{0\}\). For each \(n \in \mathbb{Z}\), we have
\[u_{n} = 0 \Leftrightarrow b_{1}c_{1}^{n} = -b_{2}c_{2}^{n} \Leftrightarrow \left(\frac{c_{1}}{c_{2}}\right)^{n} = \frac{-b_{2}}{b_{1}} \Leftrightarrow C^{n} = B\]
and, for every prime \(p\) coprime to all of \(b\), \(b_{1}\), \(b_{2}\), \(c_{1}\), and \(c_{2}\) simultaneously, we have
\[u_{n} \equiv 0 \textrm{ (mod \(p\))} \Leftrightarrow b_{1}c_{1}^{n} \equiv -b_{2}c_{2}^{n} \textrm{ (mod \(p\))}\]
\[\Leftrightarrow \left(\frac{c_{1}}{c_{2}}\right)^{n} \equiv \frac{-b_{2}}{b_{1}} \textrm{ (mod \(p\))} \Leftrightarrow C^{n} \equiv B \textrm{ (mod \(p\))}.\]

Cases (1) and (2) of Theorem \ref{thmain} are straightforward. In case (1), we have \(C^{m} = B\) for some \(m \in \mathbb{Z}\), so \(u_{m} = 0\) and Theorem \ref{thmain} holds. In case (2), we have \((C, B) = (1, -1)\), so \(C^{n} \neq B\) for all \(n \in \mathbb{Z}\) and, for every odd prime \(p\), \(C^{n} \not\equiv B\) (mod \(p\)), so Theorem \ref{thmain} holds.

From now on, assume that we are in case (3); that is, assume that \(B^{k}C^{\ell} \notin \{\pm 1\}\) for all \(k \in \mathbb{Z} - \{0\}\) and all \(\ell \in \mathbb{Z}\). We must show \(C^{n} \not \equiv B \pmod{p}\) for infinitely many primes \(p\). To prepare for the proof of this case, we recall some background on power residue symbols.

\section{Power Residue Symbols}

(See \cite[pp.\@ 244--248]{M20} and \cite[exercise 2, pp.\@ 351--354]{CF10}.) Consider a number field \(K\) and a positive integer \(n\) with \(K \supseteq \mu_{n}\).

\begin{definition}[Power residue symbols]
Let \(a \in K - \{0\}\). Let \(S(a)\) be the set of archimedean places and prime ideals \(\mathfrak{p}\) such that \(\mathrm{ord}_{\mathfrak{p}}(n) \neq 0\) or \(\mathrm{ord}_{\mathfrak{p}}(a) \neq 0\). We define the \emph{\(n\)th power residue symbols} \(\left(\frac{a}{\mathfrak{b}}\right)_{n}\) and \(\left(\frac{a}{b}\right)_{n}\) as follows.
\begin{itemize}
\item[(1)] For prime ideals \(\mathfrak{p}\) such that \(\mathrm{ord}_{\mathfrak{p}}(n) = \mathrm{ord}_{\mathfrak{p}}(a) = 0\), let \(\left(\frac{a}{\mathfrak{p}}\right)_{n}\) be the \(n\)th root of 1 such that \(\left(\frac{a}{\mathfrak{p}}\right)_{n} \equiv a^{(N\mathfrak{p} - 1)/n}\) mod \(\mathfrak{p}\).
\item[(2)] Let \(\left(\frac{a}{\prod_{j}\mathfrak{p}_{j}^{c_{j}}}\right)_{n} = \prod_{j}\left(\frac{a}{\mathfrak{p}_{j}}\right)_{n}^{c_{j}}\) for integers \(c_{j}\) (for finite products).
\item[(3)] For \(b \in K - \{0\}\), let \(\left(\frac{a}{b}\right)_{n} = \left(\frac{a}{\langle{}b\rangle^{S(a)}}\right)_{n}\), where \(\langle{}b\rangle^{S(a)}\) is the ideal obtained from the prime-ideal factorisation of \(\langle{}b\rangle\) by removing all powers of prime ideals \(\mathfrak{p}\) such that \(\mathfrak{p} \in S(a)\).\hfill\(\square\)
\end{itemize}
\end{definition}

The following result of Schinzel says that, in a certain sense, we can choose the values of finitely many power residue symbols.
\begin{theorem}
\cite[Theorem 4, p.\@ 246]{S77} Let \(n\) be an odd prime. Let some numbers \(a_{1}, \ldots, a_{k}\) in \(\mathbb{Q}\) satisfy:
\begin{center}
For \(x_{1}, \ldots, x_{k} \in \mathbb{Z}\) and \(r \in \mathbb{Q}\), if \(a_{1}^{x_{1}}\cdots{}a_{k}^{x_{k}} = r^{n}\),

then \(n\) divides each of \(x_{1}, \ldots, x_{k}\).
\end{center}
It follows that for \(c_{1}, \ldots, c_{k} \in \mathbb{Z}\), there are infinitely many primes \(p\) such that, for \(i \in \{1, \ldots, k\}\), we have
\[\left(\frac{a_{i}}{p}\right)_{n} = \zeta_{n}^{c_{i}}.\]
In particular, this is the case if \(a_{1}, \ldots, a_{k}\) are distinct primes.\hfill\(\square\)
\end{theorem}
The paper of Schinzel \cite{S77} generalizes this to an arbitrary algebraic number field; we use only the case where the number field is \(\mathbb{Q}\). Schinzel's proof uses the Chebotarev density theorem.

\section{Proof of Remaining Part of Main Theorem}

We now return to the proof of Theorem \ref{thmain}. The part that we still need to prove is case (3): Under the assumptions that \(C \in \mathbb{Q} - \{0\}\), \(B \in \mathbb{Q} - \{0, 1\}\), \((C, B) \neq (1, -1)\), and \(B^{k}C^{\ell} \notin \{\pm 1\}\) for all \(k \in \mathbb{Z} - \{0\}\) and all \(\ell \in \mathbb{Z}\), we need to show that there is an infinite set \(S\) of primes, all coprime to \(B\) and to \(C\) simultaneously, such that, for all \(r\) in \(S\) and all \(k\) in \(\mathbb{Z}\), we have \(C^{k} \not\equiv B \pmod{r}\).

Let \(P\) be the set of all primes. Let \(T = \{p_{1}, \ldots, p_{m}\}\) be the set of all primes \(p\) such that
\[(v_{p}(b), v_{p}(b_{1}), v_{p}(b_{2}), v_{p}(c_{1}), v_{p}(c_{2})) \neq (0, 0, 0, 0, 0).\]
Let \(M\) be the matrix
\[\left(\begin{array}{cccc}
v_{p_{1}}(C) & v_{p_{2}}(C) & \cdots & v_{p_{m}}(C)\\
v_{p_{1}}(B) & v_{p_{2}}(B) & \cdots & v_{p_{m}}(B)
\end{array}\right) = \left(\begin{array}{c}
v_{p}(C)\\
v_{p}(B)
\end{array}\right)_{p \in T}.\]

Since \(B^{k}C^{\ell} \notin \{\pm 1\}\) for all \(k \in \mathbb{Z} - \{0\}\) and all \(\ell \in \mathbb{Z}\), it follows that
\begin{itemize}
\item[(1)] The top row of \(M\) is a zero row, or
\item[(2)] The matrix \(M\) has rank 2.
\end{itemize}

Therefore, at least one of the following two conditions holds.
\begin{itemize}
\item[(1)] Some prime \(p'\) is such that \(v_{p'}(C) = 0\) and \(v_{p'}(B) \neq 0\).
\item[(2)] Some primes \(p'\) and \(q'\) are such that
\[0 \neq d := \det\left(\begin{array}{cc}
v_{p'}(C) & v_{q'}(C)\\
v_{p'}(B) & v_{q'}(B)
\end{array}\right).\]
\end{itemize}
In the following argument, wherever a list with items (1) and (2) occurs, each item in the list applies in the case where the corresponding condition above is true.

Choose a prime \(n\) such that
\begin{itemize}
\item[(1)] \(\gcd(n, 2v_{p'}(B)) = 1\).
\item[(2)] \(\gcd(n, 2d) = 1\).
\end{itemize}
We consider the \(n\)th power residue symbol for \(\mathbb{Q}(\zeta_{n})\).

Let \(S\) be the set of primes \(r\) in \(P - T - \{n\}\) such that:
\begin{itemize}
\item[(1)] We have \(\left(\frac{p}{r}\right)_{n} = 1\) for \(p \in T - \{p'\}\) and \(\left(\frac{p'}{r}\right)_{n} = \zeta_{n}\).
\item[(2)] We have \(\left(\frac{p}{r}\right)_{n} = 1\) for \(p \in T - \{p', q'\}\), \(\left(\frac{p'}{r}\right)_{n} = \zeta_{n}^{a}\), and \(\left(\frac{q'}{r}\right)_{n} = \zeta_{n}^{b}\), where the numbers \(a, b \in \{0, 1, \ldots, n - 1\}\) are fixed so that
\[\begin{array}{r}
av_{p'}(C) + bv_{q'}(C) \equiv 0 \pmod{n}\\
\textrm{and } av_{p'}(B) + bv_{q'}(B) \equiv 1 \pmod{n}
\end{array}\]
(such \(a\) and \(b\) exist by the previous conditions).
\end{itemize}

The set \(S\) is infinite by Schinzel's result.

Note that \(n\) is odd, so \((-\zeta_{n})^{n} = (-1)^{n} = -1\), so \(-1\) is an \(n\)th power in \(\mathbb{Q}(\zeta_{n})\), so every \(r \in S\) satisfies \(\left(\frac{-1}{r}\right)_{n} = 1\).

It follows that for every \(r \in S\), the number \(\left(\frac{B}{r}\right)_{n}\) is a primitive \(n\)th root of unity, but we have \(\left(\frac{C}{r}\right)_{n} = 1\) and therefore \(\left(\frac{C^{k}}{r}\right)_{n} = 1\) for all \(k \in \mathbb{Z}\), so for every \(k \in \mathbb{Z}\), we have \(C^{k} \not\equiv B\) (mod \(r\)). This completes the proof of Theorem \ref{thmain}.

\section{Corollary: Some Degenerate Cubic Cases}

Theorem \ref{thmain} easily implies some degenerate cubic cases of the exponential local-global principle.

\begin{corollary}
Let \(c_{1}\) and \(c_{2}\) be nonzero rational numbers such that \(c_{1} \notin \{c_{2}, -c_{2}\}\). Let \(b_{1}\), \(b_{2}\), and \(b_{3}\) be nonzero rational numbers.

Consider the sequence \(u = (u_{n})_{n = -\infty}^{\infty}\) such that for \(n \in \mathbb{Z}\),
\[u_{n} = b_{1}c_{1}^{n} + (b_{2} + (-1)^{n}b_{3})c_{2}^{n}.\]
(Note that if we let the rational numbers \(a_{1}\), \(a_{2}\), and \(a_{3}\) be such that
\[(x - c_{1})(x - c_{2})(x + c_{2}) = x^{3} - a_{1}x^{2} - a_{2}x - a_{3}\]
identically, then
\[u_{n + 3} = a_{1}u_{n + 2} + a_{2}u_{n + 1} + a_{3}u_{n}\]
for all \(n \in \mathbb{Z}\).)

Let \(C = c_{1}/c_{2}\), \(B_{+} = -(b_{2} + b_{3})/b_{1}\), and \(B_{-} = -(b_{2} - b_{3})/b_{1}\).

The exponential local-global principle holds for \(u\) if, for each \(B \in \{B_{+}, B_{-}\}\), at least one of the following is true.
\begin{itemize}
\item[(1)] We have \((C, B) = (1, -1)\).
\item[(2)] We have \(B^{k}C^{\ell} \notin \{\pm 1\}\) for all \(k \in \mathbb{Z} - \{0\}\) and all \(\ell \in \mathbb{Z}\).
\end{itemize}
\end{corollary}
\begin{proof}
Apply Theorem \ref{thmain} to the sequences \((u_{2n})_{n = -\infty}^{\infty}\) and \((u_{2n + 1})_{n = -\infty}^{\infty}\).
\end{proof}

\section*{Acknowledgments}

Many thanks to Eder Kikianty and James Raftery for their mentorship and support for this project.

Many thanks to Mapundi Banda, the head of the Department of Mathematics and Applied Mathematics at the University of Pretoria, for their support.

This project was awarded funding from the University of Pretoria's Research Development Programme. Many thanks to everyone at the University of Pretoria for their support.

\end{document}